\documentclass{amsart}

\usepackage{amsfonts}
\usepackage{amssymb}
\usepackage{enumerate}
\usepackage{amsmath}
\usepackage{amsthm}
\usepackage{graphicx}
\usepackage[english]{babel}

\numberwithin{equation}{section}

\newtheorem{theorem}{Theorem}[section]
\newtheorem{lemma}[theorem]{Lemma}

\newtheorem{problem}[theorem]{Problem}

\newtheorem{corollary}[theorem]{Corollary}
\newtheorem{statement}[theorem]{Statement}
\newtheorem{fact}[theorem]{Fact}
\theoremstyle{definition}
\newtheorem{example}[theorem]{Example}
\newtheorem{remark}[theorem]{Remark}
\newtheorem{definition}[theorem]{Definition}
\newtheorem{notation}[theorem]{Notation}

\DeclareMathOperator{\diam}{diam}
\DeclareMathOperator{\dist}{dist}

\DeclareMathOperator{\graph}{graph}
\DeclareMathOperator{\inter}{int}

\DeclareMathOperator{\cl}{cl}

\DeclareMathOperator{\Prob}{Pr}

\newcommand{\NN}{\mathbb{N}}

\newcommand{\RR}{\mathbb{R}}
\newcommand{\ZZ}{\mathbb{Z}}
\newcommand{\PP}{\mathbb{P}}

\newcommand{\EE}{\mathbb{E}}

\newcommand{\iA}{\mathcal{A}}
\newcommand{\iB}{\mathcal{B}}

\newcommand{\iH}{\mathcal{H}}
\newcommand{\iI}{\mathcal{I}}

\newcommand{\iF}{\mathcal{F}}

\newcommand{\iS}{\mathcal{S}}

\begin{document}

\title{Dimensions of graphs of prevalent continuous maps}

\author{Rich\'ard Balka}
\address{Current address: Department of Mathematics, University of British Columbia, and Pacific Institute for the Mathematical Sciences, Vancouver,
BC V6T 1Z2, Canada}
\email{balka@math.ubc.ca}
\address{Former affiliations: Department of Mathematics, University of Washington, Box 354350, Seattle, WA 98195-4350, USA and Alfr\'ed R\'enyi Institute of Mathematics, Hungarian Academy of Sciences, PO Box 127, 1364 Budapest, Hungary}

\thanks{The author was supported by the Hungarian Scientific Research Fund grant no.~104178.}

\subjclass[2010]{Primary: 28A78, 28C10, 46E15, Secondary: 60B05, 54E52.}

\keywords{Haar null, shy, prevalent, Hausdorff dimension, packing dimension, Minkowski dimension, box dimension, graph, continuous map.}

\begin{abstract}
Let $K$ be an uncountable compact metric space and let $C(K,\mathbb{R}^d)$ denote the set of continuous maps $f\colon K \to \mathbb{R}^d$ endowed with the maximum norm. The goal of this paper is to determine various fractal dimensions of the graph of the prevalent $f\in C(K,\mathbb{R}^d)$.

As the main result of the paper we show that if $K$ has finitely many isolated points then
the lower and upper box dimension of the graph of the prevalent $f\in C(K,\mathbb{R}^d)$ are
$\underline{\dim}_B K+d$ and $\overline{\dim}_B K+d$, respectively.
This generalizes a theorem of Gruslys, Jonu\v{s}as, Mijovi\`c, Ng, Olsen, and Petrykiewicz.

We prove that the graph of the prevalent $f\in C(K,\mathbb{R}^d)$ has packing dimension $\dim_P K+d$,
generalizing a result of Balka, Darji, and Elekes.

Balka, Darji, and Elekes proved that the Hausdorff dimension of the graph of the prevalent
$f\in C(K,\mathbb{R}^d)$ equals $\dim_H K+d$.
We give a simpler proof for this statement based on a method of Fraser and Hyde.
\end{abstract}

\maketitle
\thispagestyle{empty}

\section{Introduction}

Assume that $G$ is a \emph{Polish group}, that is, a separable topological group endowed with a compatible complete metric.
If $G$ is locally compact then it admits a \emph{Haar measure},
i.e.\ a left translation invariant Borel measure which is regular, finite on compact sets,
and positive on non-empty open sets.
The concept of Haar measure cannot be extended to groups that are not locally compact, but surprisingly the idea of Haar measure zero sets can. The next
definition is due to Christensen \cite{C}, which was rediscovered later by Hunt, Sauer and York \cite{HSY}.

\begin{definition} \label{d:shy} Let $G$ be an abelian Polish group and let $A\subset G$. Then $A$ is called \emph{shy} or \emph{Haar null}
if there exists a Borel set $B\subset G$ and a Borel probability measure $\mu$ on $G$ so that $A\subset B$ and $\mu\left(B+x\right) = 0$ for all $x\in G$.
The complement of a shy set is a \emph{prevalent} set.
\end{definition}

Shy sets form a $\sigma$-ideal, and in a locally compact abelian Polish group
they coincide with the Haar measure zero sets, see \cite{C}.
We will apply this concept for the Banach space $G=C(K,\RR^d)$.

\begin{notation} The Hausdorff, lower box, upper box and packing dimension of a metric space $X$ is denoted by
$\dim_H X$, $\underline{\dim}_B X$, $\overline{\dim}_B X$, and $\dim_P X$, respectively.
We use the convention $\dim \emptyset=-1$ for each of the above dimensions. We simply write $C[0,1]=C([0,1],\RR)$.
\end{notation}

Over the last three decades there has been a huge interest in studying properties of \emph{typical} objects,
where typical might mean both generic in the sense of Baire category and prevalent.
Now we summarize the results on dimensions of graphs of continuous maps.
In the category setting Mauldin and Williams \cite{MW} proved the following.

\begin{theorem}[Mauldin-Williams] \label{t:MW} For a generic $f\in C[0,1]$ we have
$$\dim_H \graph(f)=1.$$
\end{theorem}

Indeed, the strategy of Mauldin and Williams easily implies the following result, see also \cite{BBE}.

\begin{theorem}
Let $K$ be an uncountable compact metric space and let $d\in \mathbb{N}^+$.
Then for a generic $f\in C(K, \RR^d)$ we have
\[
\dim_H \graph(f)=\dim_H K.
\]
\end{theorem}

The following theorems were proved by Hyde et al.\ \cite{HLOPS}. In fact, they considered the case $K\subset \RR$ and $d=1$,
but their proof easily yields the following theorems.

\begin{theorem}[Hyde et al.]
Let $K$ be a compact metric space and let $d\in \NN^+$. Then for a generic $f\in C(K,\RR^d)$ we have
$$\underline{\dim}_B \graph(f)= \underline{\dim}_B K.$$
\end{theorem}

\begin{theorem}[Hyde et al.] \label{t:Hy}
Let $K$ be an uncountable compact metric space with at most finitely many isolated points and let $d\in \NN^+$.
Then for a generic $f\in C(K,\RR^d)$ we have
$$ \overline{\dim}_B \graph(f)= \overline{\dim}_B K+d.$$
\end{theorem}

The following result was proved by Humke and Petruska \cite{HP}.

\begin{theorem}[Humke-Petruska] For a generic $f\in C[0,1]$ we have
$$\dim_P \graph(f)=2.$$
\end{theorem}

Recently Liu et al.\ \cite{LTW} proved the following generalization.

\begin{theorem}[Liu et al.] Let $K$ be an uncountable compact metric space. Then for a generic $f\in C(K,\RR)$ we have
$$\dim_P \graph(f)=\dim_P K+1.$$
\end{theorem}

Now consider graphs of prevalent continuous maps.
First McClure proved in \cite{Mc} that the packing dimension (and hence the upper box dimension) of the graph of a prevalent
$f\in C[0,1]$ is 2. For the lower box dimension the analogous result was proved independently in \cite{FF}, \cite{GJMNOP}, and \cite{Sh}.
Moreover, Gruslys et al.\ \cite{GJMNOP} proved the following theorem.

\begin{theorem}[Gruslys et al.] \label{t:GJMNOP}
Let $K\subset \RR^m$ be an uncountable compact set. Assume that $K$ satisfies the following property: there is a $\delta_0>0$ such that
for all $\delta\leq \delta_0$ and for every cube of the form $Q=\prod_{i=1}^{m}[m_i\delta,(m_i+1)\delta]$ $(m_i\in \ZZ)$
the intersection $K\cap Q$ is path connected. Then
for a prevalent $f\in C(K,\RR)$ we have
\begin{align*} \underline{\dim}_B\graph(f)&=\underline{\dim}_B K+1, \\
\overline{\dim}_B \graph(f)&=\overline{\dim}_B K+1.
\end{align*}
\end{theorem}

As the main result of this paper, we generalize Theorem~\ref{t:GJMNOP} in Section~\ref{s:M}.

\begin{theorem} \label{t:grM} Let $K$ be an uncountable compact metric space with at most finitely many isolated points and let $d\in \NN^+$.
Then for a prevalent $f\in C(K,\RR^d)$ we have
\begin{align*} \underline{\dim}_B\graph(f)&=\underline{\dim}_B K+d, \\
\overline{\dim}_B \graph(f)&=\overline{\dim}_B K+d.
\end{align*}
\end{theorem}

In the proof of Theorem~\ref{t:GJMNOP} prevalence is witnessed by a measure supported on a one-dimensional subspace,
see also Theorem~\ref{t:Gea} and the subsequent discussion.
The proof (in the case of upper box dimension) uses
that if $K$ satisfies the connectivity condition of Theorem~\ref{t:GJMNOP} then for all $f,g\in C(K,\RR)$ we have
\begin{equation*}
\overline{\dim}_B \graph(f+g)\leq \max\{\overline{\dim}_B \graph(f),\overline{\dim}_B \graph(g)\}.
\end{equation*}
The next theorem shows that the above inequality is not true even for the triadic Cantor set $K$.
That is why in the proof of Theorem~\ref{t:grM} prevalence will be witnessed by a more complicated `infinite dimensional' measure.

\begin{theorem} \label{t:Cantor} Let $K\subset [0,1]$ be the triadic Cantor set. Then
there exist functions $f,g\in C(K,\RR)$ such that
\begin{equation*} \overline{\dim}_B \graph(f+g)> \max\{\overline{\dim}_B \graph(f),\overline{\dim}_B \graph(g)\}.
\end{equation*}
\end{theorem}

Note that if $K$ has infinitely many isolated points, then Theorem~\ref{t:grM} may not hold. For the following example see \cite{HLOPS}.

\begin{example}[Hyde et al.] \label{ex1} Let $K=\{0\}\cup \{1/n: n\in \NN^+\}$. Then
$$\sup_{f\in C(K,\RR)} \overline{\dim}_B \graph(f)\leq 1<3/2=\underline{\dim}_B K+1.$$
\end{example}

For packing dimension Balka, Darji, and Elekes \cite{BDE} proved the following.

\begin{theorem}[Balka-Darji-Elekes] \label{t:BDE} Assume that $m,d\in \NN^+$ and $K\subset \RR^m$ is
an uncountable compact set. Then for a prevalent $f\in C(K,\RR^d)$ we have
$$\dim_P \graph(f)=\dim_P K+d.$$
\end{theorem}

In Section~\ref{s:P} we generalize the above theorem based on Theorem~\ref{t:grM}.

\begin{theorem} \label{t:grP} Let $K$ be an uncountable compact metric space and let $d\in \NN^+$.
Then for a prevalent $f\in C(K,\RR^d)$ we have
$$\dim_P \graph(f)=\dim_P K+d.$$
\end{theorem}

Fraser and Hyde \cite{FH} showed that the graph of a prevalent $f\in C[0,1]$ has maximal Hausdorff dimension.
This improves the analogous results concerning box and packing dimension, see Fact~\ref{f:<}.

\begin{theorem}[Fraser-Hyde] \label{t:FH} For a prevalent $f\in C[0,1]$ we have
$$\dim_H \graph(f)=2.$$
\end{theorem}

The following generalization is due to Bayart and Heurteaux \cite{BH}.

\begin{theorem}[Bayart-Heurteaux] \label{t:BH} Let $K\subset \RR^m$ be compact with $\dim_H K>0$. Then for a prevalent $f\in C(K,\RR)$ we have
$$\dim_H \graph(f)=\dim_H K+1.$$
\end{theorem}

\begin{remark} Recently Peres and Sousi \cite{PS} proved a stronger result for compact sets $K\subset \RR$.
Let $X\colon K\to \RR^d$ be a fractional Brownian motion restricted to $K$ and let $f\in C(K,\RR^d)$ be given.
In \cite{PS} the almost sure Hausdorff dimension of $\graph(X+f)$ is determined in terms of $f$ and the Hurst index of $X$.
\end{remark}

The proof of Theorem~\ref{t:BH} is based on the energy method, see \cite[Theorem~3]{BH}. A lower estimate for the Hausdorff
dimension of $\graph(X+f)$ is given there, where $X\colon K\to \RR$ is a fractional Brownian motion restricted to $K\subset \RR^m$ and
$f\in C(K,\RR)$ is a continuous drift. In fact, the proof easily extends to vector valued functions,
and (as pointed out in \cite{BFFH}) Dougherty's result on images handles the case $\dim_H K=0$, see Theorem~\ref{t:D}.
These yield the following theorem.

\begin{theorem}\label{t:gr} Assume that $m,d\in \NN^+$ and $K\subset \RR^m$ is an uncountable compact set.
Then for a prevalent $f\in C(K,\RR^d)$ we have
$$\dim_H \graph(f)=\dim_H K+d.$$
\end{theorem}

Balka, Darji, and Elekes proved in \cite{BDE} that the condition $K\subset \RR^m$ is superfluous.

\begin{theorem}[Balka-Darji-Elekes] \label{t:H}
Let $K$ be an uncountable compact metric space and let $d\in \NN^+$.
Then for a prevalent $f\in C(K,\RR^d)$ we have
$$\dim_H \graph(f)=\dim_H K+d.$$
\end{theorem}

In \cite{BDE} the above theorem is a corollary of a much deeper result concerning the fibers of a prevalent
$f\in C(K,\RR^d)$. Following Fraser and Hyde \cite{FH}, in Section~\ref{s:H}
we give a simpler proof for Theorem~\ref{t:H} based on the energy method.

Finally, in Section~\ref{s:open} we pose some open problems.

\section{Preliminaries}
Probability and expectation will be denoted by $\Pr$ and $\EE$, and $|\cdot|$ denotes absolute value.
The compact metric space $K$ is called a \emph{Cantor space} if it is perfect and totally disconnected.
Let $(X,\rho)$ be a metric space. We endow $X\times \RR^d$ by the metric
$$\rho_{X\times \RR^d}((x_1,z_1),(x_2,z_2))=\sqrt{\rho(x_1,x_2)^2+|z_1-z_2|^2}.$$
For $x\in X$ and $r>0$ let $B(x,r)$ and $U(x,r)$ denote the closed and open ball of
radius $r$ centered at $x$, respectively. For $A,B \subset X$ let us define $\dist(A,B) = \inf\{\rho(x,y) : x\in A,~y\in B\}$. Let
$\diam A$, $\inter A$, and $\cl A$ denote the diameter, interior, and closure of $A$, respectively.
Given $\delta>0$ we say that a set $S\subset X$ is a
\emph{$\delta$-packing} if $\rho(x,z)>\delta$ for all distinct $x,z\in S$. For $n\in \NN^+$ define
$$N_{n}(X)=\max\{\#S: S\subset X \textrm{ is a $2^{-n}$-packing} \}.$$
If $X$ is non-empty and totally bounded then the \emph{lower} and \emph{upper box dimension} of $X$ are respectively defined as
\begin{align*}
\underline{\dim}_B X &=\liminf_{n\to \infty} \frac{\log N_{n}(X)}{n\log 2}, \\
\overline{\dim}_B X&=\limsup_{n\to \infty} \frac{\log N_{n}(X)}{n\log 2}.
\end{align*}
Let $\underline{\dim}_B X=\overline{\dim}_B X=\infty$ if $X$ is not totally bounded.
The \emph{packing dimension} of $X$ is defined as
$$\dim_P X=\inf\left\{\sup_{i}\overline{\dim}_B A_i: X\subset \bigcup_{i=1}^{\infty} A_i\right\}.$$
For the following lemma see \cite[Lemma~3.2]{MM} or \cite[Lemma~4]{FaH}.

\begin{lemma}\label{l:pack1} Let $K$ be a compact metric space and let $s\in \RR$.
If $\dim_P K>s$ then there is a compact set $C\subset K$ such that $\dim_P (C\cap U)\geq s$ for all open sets $U$ with $C\cap U\neq \emptyset$.
\end{lemma}

For the following lemma see the proof of \cite[Proposition~3]{T} or \cite[Corollary~3.9]{F}.

\begin{lemma} \label{l:pack2} Let $K$ be a compact metric space and $s\in \RR$.
If $\overline{\dim}_{B} (K\cap U)\geq s$ for every non-empty open set $U\subset K$, then
$\dim_P K\geq s$.
\end{lemma}

For $s \geq 0$ the \emph{$s$-dimensional Hausdorff content} of $X$ is defined as
\begin{equation*}
\mathcal{H}_{\infty}^{s}(X)=\inf \left\{ \sum_{i=1}^\infty (\diam
A_{i})^{s}: X \subset \bigcup_{i=1}^{\infty} A_{i} \right\}.
\end{equation*}
The \emph{Hausdorff dimension} of a non-empty $X$ is defined as
$$\dim_{H} X= \inf\{s \geq 0: \mathcal{H}_{\infty}^{s}(X) =0\}.$$
Recall that $\dim \emptyset=-1$ by convention for each of the above dimensions.
For a Borel probability measure $\nu$ on $X$ and $s>0$ we define the $s$-energy of $\nu$ by
$$I_s(\nu)=\iint_{X^2} \frac{\mathrm{d} \nu(x) \, \mathrm{d} \nu(y)}{\rho(x,y)^{s}}.$$
For the following theorem see \cite[Theorem~8.9]{Ma} and Frostman's lemma for compact metric spaces \cite[Theorem~8.17]{Ma}.
\begin{theorem} \label{t:energy} For a compact metric space $K$ we have
$$\dim_H K=\sup\{s>0: \exists \, \nu \textrm{ on } K \textrm{ such that } I_s(\nu)<\infty\}.$$
\end{theorem}

For the following facts and for more on these concepts see \cite{F}.

\begin{fact} \label{f:<} For any metric space $X$ we have
$$\dim_H X\leq \underline{\dim}_B X\leq \overline{\dim}_B X  \quad \textrm{and} \quad \dim_H X\leq \dim_P X\leq \overline{\dim}_B X.$$
\end{fact}

\begin{fact} \label{f:prod} Let $\dim$ be one of the above dimensions. Then for every non-empty
metric space $X$ and $d\in \NN^+$ we have
\begin{equation*} \dim (X\times [0,1]^d)=\dim X+d.
\end{equation*}
\end{fact}

The next lemma is \cite[Proposition~8]{D}.

\begin{lemma}\label{l:D} Assume that $G_1,G_2$ are abelian Polish groups and $\Phi \colon G_1\to G_2$ is a continuous onto homomorphism.
If $S\subset G_2$ is prevalent then so is $\Phi^{-1}(S)\subset G_1$.
\end{lemma}

Lemma~\ref{l:D} and Tietze's extension theorem in $\mathbb{R}^d$ imply the following corollary.

\begin{corollary} \label{c:her}
Assume that $K_1\subset K_2$ are compact metric spaces and $d\in \mathbb{N}^+$. Let
$$R\colon C(K_2,\mathbb{R}^d)\to C(K_1,\mathbb{R}^d), \quad R(f)=f|_{K_1}.$$
If $\mathcal{A} \subset C(K_1,\mathbb{R}^d)$ is prevalent then $R^{-1}(\mathcal{A})\subset C(K_2,\mathbb{R}^d)$ is prevalent, too.
\end{corollary}

\section{Upper and lower box dimensions} \label{s:M}

The aim of this section is to prove Theorems~\ref{t:grM} and \ref{t:Cantor}.

\begin{proof}[Proof of Theorem~\ref{t:grM}]
We may remove the finitely many isolated points from $K$ without changing the lower and upper box dimensions of the set.
This and Corollary~\ref{c:her} yield that we may assume that $K$ is perfect. By Fact~\ref{f:prod} it is enough to show only the lower bounds.
That is, we need to prove that for a prevalent $f\in C(K,\RR^d)$ we have
\begin{equation} \label{eq:gr} \underline{\dim}_B\graph(f)\geq \underline{\dim}_B K+d \quad \textrm{and} \quad
\overline{\dim}_B \graph(f)\geq\overline{\dim}_B K+d.
\end{equation}
For every $n\in \NN^+$ define the open set
$$\iA_n=\{f\in C(K,\RR^d): N_{n}(\graph(f))\geq N_{n}(K)2^{nd}n^{-2d}\},$$
where recall that $N_n(X)$ denotes the cardinality of the maximal $2^{-n}$-packing in $X$.
In order to show \eqref{eq:gr} it is enough to prove that the set
$$\iA:=\liminf_{n} \iA_n=\bigcup_{k=1}^{\infty} \left(\bigcap_{n=k}^{\infty} \iA_n\right)$$
is prevalent. As $\iA_n$ are open, $\iA$ is Borel.
We need to construct a Borel probability measure $\mu$ on $C(K,\RR^d)$ such that $\mu(A-g)=1$ for all $g\in C(K,\RR^d)$.

First we define $\mu$. For all $n\in \NN^+$ let us define $S_n\subset \RR^d$ as
$$S_n=2^{-n+3}\{0,1,\dots, \lfloor 2^{n}n^{-2} \rfloor\}^{d},$$
where $ \lfloor x  \rfloor$ denotes the integer part of $x$. Then clearly $S_n$ is a $2^{-n+2}$-packing such that $\#S_n\geq 2^{nd}n^{-2d}$. For all $n\in \NN^+$ let
$$s_n=\#S_n \quad \textrm{and} \quad k_n=N_{n}(K).$$
Let $\{X^n_i\}_{i,n\geq 1}$ be independent random variables defined on a measurable space $(\Omega,\iF)$
such that $\{X^n_i\}_{i\geq 1}$
is an i.i.d. sequence for each $n\in \NN^+$ with distribution given by
$$\Prob(X_i^n=y)=\frac{1}{s_n} \quad \textrm{for all } y\in S_n.$$
For each $n,i\geq 1$ define the generated $\sigma$-algebra
$$\iF_i^{n}=\sigma(X_{j}^n: 1\leq j\leq i-1),$$
where $\iF_1^{n}=\{\emptyset, \Omega\}$.
\begin{statement} \label{st} Let $n\in \NN^+$. There exists $\ell_n\in \NN^+$ with the following property.
For every random sequence $\{y_i\}_{i\geq 1}$ in $\RR^d$ such that $y_i$ is $\iF_i^{n}$-measurable we have
$$\Prob\left(N_{n}\left(\bigcup_{i=1}^{\ell_n} \{X^n_i+y_i\}\right)<s_n\right)\leq \frac{1}{k_n2^n}.$$
\end{statement}
\begin{proof}[Proof of Statement~\ref{st}]
Define $\ell_n=s_n m_n$, where $m_n\in \NN^+$ is so large that
\begin{equation} \label{eq:mn} \left(1-\frac{1}{s_n}\right)^{m_n}\leq \frac{1}{s_n k_n 2^n}.
\end{equation}
Fix an arbitrary random sequence $\{y_i\}_{i\geq 1}$. For all $s\in \{1,\dots,s_n\}$ let
$$Z_s=\bigcup_{i=1}^{sm_n} \{X^n_i+y_i\} \quad \textrm{and} \quad  N(s)=N_{n}(Z_{s}).$$
Now we prove by induction that for all $s\in \{1,\dots,s_n\}$ we have
\begin{equation} \label{eq:s}
\Prob\left(N(s)<s\right)\leq \frac{s}{s_nk_n2^n},
\end{equation}
and the case $s=s_n$ will complete the proof. If $s=1$ then \eqref{eq:s} is straightforward. For the induction step we need to prove that
\begin{equation} \label{eq:ind} \Prob(N(s+1)<s+1)-\Prob(N(s)<s)=\Prob(N(s+1)=N(s)=s)\leq \frac{1}{s_n k_n 2^n}.
\end{equation}
Suppose that $N(s)=s$ and $i\in \{sm_n+1,\dots, (s+1)m_n\}$ is fixed. First we prove that there is an $\iF_i^{n}$-measurable random $x_i\in S_n$ such that
\begin{equation} \label{eq:xi} N_{n}(Z_{s}\cup \{x_i+y_i\})=s+1.
\end{equation}
Indeed, the distance between any two balls of $\{B(x,2^{-n})\}_{x\in S_n}$ is at least $2^{-n+1}$, so
$N_{n}(Z_{s}-y_i)=N_{n}(Z_{s})=s<s_n$ implies that there is an $x_i\in S_n$ such that
$B(x_i,2^{-n})\cap (Z_{s}-y_i)=\emptyset$. Thus $\dist(Z_{s},\{x_i+y_i\})>2^{-n}$, so
\eqref{eq:xi} holds. As $x_i$ depends only on $y_i$ and $Z_s$, it is clearly $\iF_i^{n}$-measurable. Let
$B_i$ be the event that $X_j^{n} \neq x_j$ for all $sm_n<j<i$, then $B_i\in \iF_i^{n}$. As $x_i$ is $\iF_i^{n}$-measurable and
$X_i^{n}$ is independent of $\iF_i^{n}$, we have
\begin{equation} \label{eq:Bi} \Prob(X_i^{n}\neq x_i \,|\, B_i)=\Prob(X_i^{n}\neq x_i)=1-\frac{1}{s_n}.
\end{equation}
Therefore \eqref{eq:xi}, \eqref{eq:Bi}, and \eqref{eq:mn} imply that
\begin{align*}
\Prob(N(s+1)=N(s)=s)&\leq \Prob\left(X_i^{n} \neq x_i \textrm{ for all } sm_n<i\leq (s+1)m_n\right)\\
&=\prod_{i=sm_n+1}^{(s+1)m_n} \Prob(X_i^{n}\neq x_i \,|\, B_i) \\
&=\left(1-\frac{1}{s_n}\right)^{m_n} \leq \frac{1}{s_n k_n 2^n}.
\end{align*}
Thus \eqref{eq:ind} holds, and the proof of the statement is complete.
\end{proof}

Now we return to the proof of Theorem~\ref{t:grM}. For all $n$ let $\{x^{n}_k\}_{1\leq k\leq k_n}$ be a $2^{-n}$-packing in $K$ and assume that
for some $\varepsilon_n>0$ for all $j\neq k$ we have
\begin{equation} \label{eq:rho} \rho(x^n_j,x^n_k)\geq 2^{-n}+3\varepsilon_n,
\end{equation}
where $\rho$ denotes the metric of $K$.
As $K$ is perfect, for each $n\in \NN^+$ and $k\in \{1,\dots,k_n\}$
we can define distinct points $\{x^n_{k,i}\}_{1\leq i\leq \ell_n}$ in $B(x^n_k,\varepsilon_n)$
such that
$$E_n=\bigcup_{k=1}^{k_n}\bigcup_{i=1}^{\ell_n} \{x^n_{k,i}\}$$
satisfy
\begin{equation} \label{eq:En} E_m\cap E_n=\emptyset \quad \textrm{for all} \quad m<n.
\end{equation}
Let us define the random function $f_n\colon E_n\to \RR^d$ such that
$$f_n(x^n_{k,i})=X^n_i.$$
Tietze's extension theorem for the coordinate functions and \eqref{eq:En} imply
that the sample functions $f_n=f_n(\omega)$ can be extended to $f_n\in C(K,\RR^d)$ such that
\begin{enumerate}
\item \label{en:2} $f_n(x)=0$ if $x\in E_m$ for some $m<n$;
\item \label{en:3} $f_n(x)\in 8n^{-2}[0,1]^d$ for all $x\in K$.
\end{enumerate}
Let $\PP_n$ be the probability measure on $C(K,\RR^d)$ corresponding to
this method of randomly choosing $f_n$, and let $\iB_n\subset C(K,\RR^d)$ be its finite support. Clearly we have
$\#\iB_n=s_n^{\ell_n}$ and $\PP_n(\{f_n\})=s_n^{-\ell_n}$ for all $f_n\in \iB_n$. By \eqref{en:3} the sum $\sum_{n=1}^{\infty} f_n$
converges for all $f_n\in \iB_n$. Let $\mathbb{P}=\prod_{n=1}^{\infty} \mathbb{P}_n$ be a probability measure on the Borel subsets of $\iB=\prod_{n=1}^{\infty}\iB_n$ and let
$$\pi\colon \iB \to C(K,\RR^d), \quad  \pi((f_n))=\sum_{n=1}^{\infty} f_n.$$
Let us define
$$\mu=\mathbb{P}\circ \pi^{-1}.$$

Let $g\in C(K,\RR^d)$ be arbitrarily fixed, now we prove that $\mu(\iA-g)=1$. We need to show that
$\mu(\limsup_n (A^c_n-g))=0$, where $\iA_n^c$ denotes the complement of $\iA_n$. By the Borel-Cantelli lemma it is enough to prove that
$$\sum_{n=1}^{\infty} \mu(\iA^c_n-g)<\infty.$$
Fix $n\in \NN^+$, it is enough to show that $\mu(\iA^c_n-g)\leq 2^{-n}$. Let $h=g+\sum_{i=1}^{\infty} f_i$, we need to prove that
\begin{equation} \label{eq:f+g} \PP(h\notin \iA_n)\leq 2^{-n}.
\end{equation}

Let $h_0=g$ and for all $m\in \NN^+$ let $h_m=g+\sum_{i=1}^{m} f_i$. For each $k\in \{1,\dots,k_n\}$ and $i\in \{1,\dots,\ell_n\}$ define
$$y^n_{k,i}=h_{n-1}(x^n_{k,i}).$$
Fix $k\in \{1,\dots, k_n\}$. As $y^n_{k,i}$ is $\iF_i^{n}$-measurable for all $i\geq 1$,
Statement~\ref{st} yields that
$$\PP\left(N_n\left(\bigcup_{i=1}^{\ell_n} \{X^n_i+y^n_{k,i}\}\right)<s_n\right)\leq \frac{1}{k_n2^n}.$$
As $h_n(x^n_{k,i})=X^n_i+y^n_{k,i}$, summing the above inequality from $k=1$ to $k_n$ yields that
$$\PP\left(\exists k\leq k_n: N_n\left(\bigcup_{i=1}^{\ell_n} \{h_n(x^n_{k,i})\} \right)<s_n\right)\leq 2^{-n}.$$
By \eqref{eq:rho} all $k,k'\in \{1,\dots, k_n\}$ with $k\neq k'$ and $i,j\in \{1,\dots, \ell_n\}$ we have
$$\rho(x^n_{k,i}, x^n_{k',j})\geq 2^{-n}+\varepsilon_n>2^{-n}.$$
Therefore
$$\PP(N_n (\graph(h_n|_{E_n}))<k_n s_n)\leq 2^{-n}.$$
Property \eqref{en:2} yields that $h_n(x)=h(x)$ for all $x\in E_n$. As $k_n s_n\geq N_n(K)2^{nd}n^{-2d}$, we have
$$\PP(h\notin \iA_n)\leq \PP( N_n(\graph(h))<k_n s_n)\leq 2^{-n}.$$
Therefore \eqref{eq:f+g} holds, and the proof is complete.
\end{proof}

\begin{proof}[Proof of Theorem~\ref{t:Cantor}]
We may assume by scaling that
$$K=\left\{\sum_{i=1}^{\infty} a_i3^{-i}:~ a_i\in \{0,1\} \textrm{ for all } i\geq 1\right\}.$$
Define $f,g\colon K\to \RR$ such that if $x=\sum_{i=1}^{\infty} a_i3^{-i}\in K$ then
$$f(x)=\sum_{i=1}^{\infty} a_{2i-1}3^{-i} \quad \textrm{and} \quad g(x)=\sum_{i=1}^{\infty} a_{2i}3^{-i}.$$
The squares of the form $[k9^{-n},(k+1)9^{-n})\times [m9^{-n},(m+1)9^{-n})$ where $k,m\in \ZZ$ are called
the \emph{$9^{-n}$-mesh squares}. For a non-empty bounded set $X\subset \RR^2$ let $M_n(X)$ denote the number of $9^{-n}$-mesh
squares that intersect $X$. It is easy to show that
\begin{equation} \label{eq:X} \overline{\dim}_B X=\limsup_{n\to \infty} \frac{\log M_{n}(X)}{n \log 9},
\end{equation}
see also \cite[Section~3.1]{F}. For $I\subset \NN^+$ let $2^{I}$ denote
the set of functions $h\colon I \to \{0,1\}$, and for all $h\in 2^I$ let $x_h=\sum_{i\in I} h(i)3^{-i}$. First we prove that
\begin{equation} \label{eq:BP} \overline{\dim}_B \graph(f+g)=\frac 12+\frac{\log 2}{\log 3}>1.
\end{equation}
Fix $n\in \NN^+$ and let $I=\{1,\dots,2n\}$. For all $h\in 2^{I}$ and $k\in \{0,\dots ,3^n\}$ define
$$Q_{h,k}=[x_h,x_h+3^{-2n})\times [(f+g)(x_h)+k3^{-2n},(f+g)(x_h)+(k+1)3^{-2n}).$$
Clearly $Q_{h,k}$ are distinct $9^{-n}$-mesh squares. As $K+K=[0,1]$, the function
$f+g$ maps $K\cap [x_h,x_h+3^{-2n})$ onto $[(f+g)(x_h),(f+g)(x_h)+3^{-n}]$. Thus all $Q_{h,k}$ intersect $\graph(f+g)$,
and the union of $Q_{h,k}$ covers $\graph(f+g)$. Hence
$$M_n(\graph(f+g))=\#\{Q_{h,k}: h\in 2^{I},\, 0\leq k\leq 3^n\}=2^{2n}(3^n+1),$$
so \eqref{eq:X} yields \eqref{eq:BP}. Now we show that
\begin{equation} \label{eq:89} \overline{\dim}_B \graph(f)=\overline{\dim}_B \graph(g)=\frac{\log 8}{\log 9}<1.
\end{equation}
We prove this only for $f$, the proof for $g$ is analogous.
Fix $n\in \NN^+$ and define
$$J=\{1,\dots,2n\}\cup \{2n+1,2n+3,\dots,4n-1\}.$$
Then $\#J=3n$. For all $h\in 2^J$ let
$$Q_h=[x_h,x_h+3^{-2n})\times [f(x_h),f(x_h)+3^{-2n}).$$
As the map $h\mapsto (x_h,f(x_h))$ is one-to-one on $2^{J}$, the sets $Q_h$ are distinct $9^{-n}$-mesh squares.
Each $Q_h$ intersects $\graph(f)$, and the union of $Q_h$ covers $\graph(f)$. Thus
$$M_n(\graph(f))=\#\{Q_h: h\in 2^{J}\}=2^{3n}.$$
Hence \eqref{eq:X} implies \eqref{eq:89}. The theorem follows from \eqref{eq:BP} and \eqref{eq:89}.
\end{proof}

\begin{remark} Using the notation of the above proof let $F,G\in C[0,1]$ such that $F|_{K}=f$, $G|_{K}=g$,
and $F,G$ are affine on the components of $(0,1)\setminus K$. Liu et al.\ \cite[Example~3.1]{LTW} pointed out that
$$\dim_P \graph(F+G)> \max\{\dim_P \graph(F), \dim_P \graph(G)\}.$$
This answers a question of Falconer and Fraser \cite[(2.6)~page~362]{FF} in the negative.
\end{remark}

\section{Packing dimension} \label{s:P}

The goal of this section is to prove Theorem~\ref{t:grP}.

\begin{proof}[Proof of Theorem~\ref{t:grP}] We can remove countably many points from $K$ without changing the packing dimension of the set,
so by \cite[Theorem~6.4]{Ke} we may assume that $K$ is perfect. Choose a sequence $s_n \nearrow \dim_P K$ and fix $n$. By Lemma~\ref{l:pack1}
there is a compact set $K_n\subset K$ such that $\dim_P (U\cap K_n)>s_n$ for every $U\subset K$ open with $U\cap K_n\neq \emptyset$.
Clearly $K_n$ is perfect.

As a countable intersection of prevalent sets is prevalent, it is enough to show that $\dim_P \graph(f)\geq s_n+d$
for a prevalent $f\in C(K,\RR^d)$. By Corollary~\ref{c:her} it is enough to prove that
$$\iA_n=\{f\in C(K_n,\RR^d): \dim_P \graph(f)\geq s_n+d\}$$
is prevalent. Let $\{U_i\}_{i\geq 1}$ be a basis of $K_n$ consisting of non-empty open sets and let $C_i=\cl(U_i)$.
We proved that $\dim_P U_i>s_n$. Therefore the definition of $K_n$ implies that for all $i\in \NN^+$ we have
$$\overline{\dim}_B C_i\geq \overline{\dim}_B U_i\geq \dim_P U_i>s_n.$$
As $K_n$ is perfect, $C_i$ are also perfect. Therefore Theorem~\ref{t:grM} yields that
$$\iB_i=\{f\in C(C_i,\RR^d): \overline{\dim}_B \graph(f)\geq s_n+d\}$$
are prevalent. For all $i\in \NN^+$ define
$$R_i\colon  C(K_n,\RR^d)\to C(C_i,\RR^d), \quad R_i(f)=f|_{C_i}.$$
By Corollary~\ref{c:her} the sets $R^{-1}_i(\iB_i)$ are prevalent in $C(K_n,\RR^d)$, so
$\bigcap_{i=1}^{\infty} R^{-1}_i(\iB_i)$ is also prevalent. Therefore it is enough to prove that
$\bigcap_{i=1}^{\infty}  R^{-1}_i(\iB_i)\subset \iA_n$. Let us fix $f\in \bigcap_{i=1}^{\infty}  R^{-1}_i(\iB_i)$, we need to show
that $f\in \iA_n$. Let $V$ be an arbitrary non-empty relatively open subset $V$ of $\graph(f)$. By Lemma~\ref{l:pack2} it is enough to prove that
$\overline{\dim}_B V \geq s_n+d$. As $\{U_i\}_{i\in \NN}$ is an open basis of $K_n$, there is an $i\in \NN^+$
such that $\graph(f|_{C_i})\subset V$. Thus $f\in R^{-1}_i(\iB_i)$ yields that $\overline{\dim}_B V \geq \overline{\dim}_B \graph(f|_{C_i})\geq s_n+d$. The proof is complete.
\end{proof}

\section{Hausdorff dimension} \label{s:H}

The goal of this section is to give an simple proof for Theorem~\ref{t:H} by following the strategy of Fraser and Hyde \cite{FH}. First we need a theorem of
Dougherty \cite[Theorem~11]{D} stating that the image of a prevalent $f\in C(K,\RR^d)$ is as large as possible.

\begin{theorem}[Dougherty] \label{t:D} Let $K$ be an uncountable compact metric space and let $d\in \mathbb{N}^+$. Then for a prevalent $f\in C(K,\mathbb{R}^d)$ we have
$$\inter f(K)\neq \emptyset.$$
\end{theorem}

\begin{remark} In fact, Dougherty proved the above theorem only for the triadic Cantor set.
As each uncountable compact metric space contains a homeomorphic copy of the triadic Cantor set (see \cite[Corollary~6.5]{K}),
Corollary~\ref{c:her} implies the more general result.
\end{remark}

The next lemma generalizes \cite[Lemma~4.1]{FH}.

\begin{lemma} \label{l:pq} Let $p,q\in (0,1]$, $d\in \NN^+$, $\theta \in \RR^d$, and $u>d/2$. Then there is a constant $c_1 \in \RR^+$
depending only on $d$ and $u$ such that
$$\int_{[0,p]^{d}} \int_{[0,p]^{d}} \frac{\mathrm{d} \alpha \, \mathrm{d} \beta}{(q^2+|\alpha-\beta+\theta|^2)^{u}}\leq c_1 p^d q^{d-2u}.$$
\end{lemma}

\begin{proof} Let $\gamma\in \RR^d$ be arbitrary. Define $\widetilde{\gamma}\in \RR^d$ such that for all $i\in \{1,\dots,d\}$
\begin{equation*} \widetilde{\gamma}_i=\begin{cases}
-1 &\textrm{ if } \gamma_i<-1, \\
\gamma_i &\textrm{ if } -1\leq \gamma_i\leq 0, \\
0 &\textrm{ if } \gamma_i>0.
\end{cases}
\end{equation*}
Then we have
\begin{equation*}\label{eq:gamma}
\int_{[0,1]^{d}} \frac{\mathrm{d} \alpha}{(q^2+p^2|\alpha+\gamma|^2)^{u}}\leq \int_{[0,1]^{d}} \frac{\mathrm{d} \alpha}{(q^2+p^2|\alpha+\widetilde{\gamma}|^2)^{u}} \leq \int_{[-1,1]^{d}} \frac{\mathrm{d} \alpha}{(q^2+p^2|\alpha|^2)^{u}}.
\end{equation*}
Applying the above inequality for $\gamma=-\beta+p^{-1}\theta$ implies that
\begin{align*}
\int_{[0,p]^{d}} \int_{[0,p]^{d}} \frac{\mathrm{d} \alpha \, \mathrm{d} \beta}{(q^2+|\alpha-\beta+\theta|^2)^{u}}
&=p^{2d} \int_{[0,1]^{d}} \int_{[0,1]^{d}} \frac{\mathrm{d} \alpha \, \mathrm{d} \beta}{(q^2+p^2|\alpha-\beta+p^{-1}\theta|^2)^{u}} \\
&\leq p^{2d} \int_{[-1,1]^{d}} \frac{\mathrm{d} \alpha}{(q^2+p^2|\alpha|^2)^{u}} \\
&\leq p^{2d}\int_{[-1,1]^{d}} \frac{\mathrm{d} \alpha}{\left(\max\{q^2,p^2|\alpha|^2\}\right)^{u}} \\
&\leq p^{2d} \int_{|\alpha|\leq q/p} \frac{\mathrm{d} \alpha}{q^{2u}}+\int_{|\alpha|\geq q/p}  \frac{\mathrm{d} \alpha}{p^{2u}|\alpha|^{2u}} \\
&\leq p^{2d} \left(c_2 \left(\frac qp \right)^d q^{-2u}+p^{-2u} \int_{q/p}^{\infty} c_3 r^{d-1-2u} \, \mathrm{d} r \right) \\
&=\left(c_2+\frac{c_3}{d-2u}\right)p^d q^{d-2u}.
\end{align*}
As $c_2,c_3\in \RR^+$ depend only on $d$, setting $c_1:=c_2+c_3/(d-2u)$ finishes the proof.
\end{proof}

For the following lemma see the proof of \cite[Lemma~4.5]{FH}.

\begin{lemma} \label{l:Borel} Let $K$ be a compact metric space, let $d\in \NN^+$ and $s\in \RR^+$. Then
$$\iA=\{f\in C(K,\RR^d): \dim_H \graph(f)\geq s\}$$
is a Borel set in $C(K,\RR^d)$.
\end{lemma}

Now we are ready to prove Theorem~\ref{t:H}.

\begin{proof}[Proof of Theorem~\ref{t:H}]
By Fact~\ref{f:prod} it is enough to prove the lower bound.

If $\dim_H K=0$ then Theorem~\ref{t:D} implies that for a prevalent $f\in C(K,\RR^d)$ we have $\inter f(K)\neq \emptyset$,
so $\dim_H f(K)=d$. As $f(K)$ is a Lipschitz image of $\graph(f)$ and Hausdorff dimension cannot increase under a Lipschitz map,
we obtain
$$\dim_H \graph(f)\geq \dim_H f(K)=d=\dim_H K+d,$$
which finishes the proof.

Thus we may assume that $\dim_H K>0$. As every uncountable compact metric space contains a Cantor space with the same Hausdorff
dimension \cite[Theorem~6.3]{K}, by Corollary~\ref{c:her} we may assume that $K$ is a Cantor space.
Fix $0<t<s<\dim_H K$, it is enough to prove that $\dim_H \graph(f)\geq t+d$ for a prevalent $f\in C(K,\RR^d)$.
As $\dim_H K>s$, by Theorem~\ref{t:energy} there exists a Borel probability measure $\nu$ on $K$ such that
$I_s(\nu)<\infty$. Then we can define inductively for all $n\in \NN^+$ integers $a_n\in \NN^+$ and for all
$(i_1,\dots,i_n)\in \iI_n:=\prod_{k=1}^n \{1,\dots, a_k\}$ non-empty compact sets $K_{i_1\dots i_n}\subset K$ such that for all
distinct indices $(i_1,\dots, i_n),(j_1,\dots,j_n)\in \iI_n$ we have

\begin{enumerate} \item $K_{i_1\dots i_n}\cap K_{j_1\dots j_n}=\emptyset$,
\item $K_{i_1\dots i_{n+1}}\subset K_{i_1\dots i_n}$ for all $i\in \{1,\dots,a_{n+1}\}$,
\item \label{ii:3} $\diam K_{i_1\dots i_n} \leq 2^{-n^2}$.
\end{enumerate}

For all $n\in \NN^{+}$ let $S_n=\{0,2^{-n}\}^{d}$, then $\#S_n=2^d$. For all
$(i_1,\dots,i_n)\in \iI_n$ define countably many independent random variables
$X_{i_1\dots i_n}$ such that for all $y\in S_n$ we have
\begin{equation} \label{eq:XY} \Pr(X_{i_1\dots i_n}=y)=2^{-d}.
\end{equation}

For each $n\in \NN^+$ and $x\in C$ there exists a unique $(i_1,\dots,i_n)\in \iI_n$ such that $x\in K_{i_1\dots i_n}$.
Define the random function $f_n\in C(K,\RR^d)$ such that
$$f_n(x)=X_{i_1\dots i_n}.$$
Let $\mathbb{P}_n$ be the probability measure on $C(K,\RR^d)$
which corresponds to the choice of $f_n$, and let $\iS_n\subset C(K,\RR^d)$ be the finite support of $\mathbb{P}_n$.
Clearly $|f_n(x)|\leq 2^{-n}$ for all $f_n\in \iS_n$ and $x\in K$, thus $\sum_{n=1}^{\infty} f_{n}$ always converges uniformly.
Let $\mathbb{P}=\prod_{n=1}^{\infty} \mathbb{P}_n$ be a probability measure on the Borel subsets of $\iS=\prod_{n=1}^{\infty}\iS_n$
and let
$$\pi\colon \iS \to C(K,\RR^d), \quad  \pi((f_n))=\sum_{n=1}^{\infty} f_n.$$
Define
$$\mu=\mathbb{P}\circ \pi^{-1}.$$
Let us fix $g\in C(K,\RR^d)$, and let $f=\sum_{n=1}^{\infty} f_n$ be a random map. Let
$$\iA=\{h\in C(K,\RR^d): \dim_H \graph(h)\geq t+d\}.$$
As $\iA$ is a Borel set by Lemma~\ref{l:Borel},
it is enough to prove that $\mu(\iA-g)=1$. Thus it is enough to show that, almost surely, $\dim_H \graph(f+g)\geq t+d$.
Define
$$F\colon K\to \graph(f+g), \quad  F(x)=(x,(f+g)(x)).$$
Let $\nu_f=\nu \circ F^{-1}$ be a random measure supported on $\graph(f+g)$. Let $\rho$ denote the metric of $K$.
\begin{statement} \label{st:0} There is a constant $c$ depending only on $s,t,d$ such that for all
$x,y\in K$, $x\neq y$ we have
$$\EE \left(\left(\rho(x,y)^{2}+|(f+g)(x)-(f+g)(y)|^2\right)^{-\frac{t+d}{2}}\right)\leq c \rho(x,y)^{-s}.$$
\end{statement}
\begin{proof}[Proof of Statement~\ref{st:0}]
Let $n=n(x,y)$ be the largest natural number $k$ such that $x,y\in C_{i_1\dots i_k}$ for some $(i_1,\dots, i_k)\in \iI_k$,
where $\max \emptyset =0$ by convention. Then \eqref{ii:3} yields that
there is a constant $c_4$ which depends only on $s,t,d$ such that
\begin{equation} \label{eq:c4} 2^{nd}\leq c_{4} \rho(x,y)^{t-s}.
\end{equation}
Clearly $$f(x)-f(y)=\sum_{i=n(x,y)+1}^{\infty} f_i(x)-\sum_{i=n(x,y)+1}^{\infty} f_i(y)=X-Y,$$
where $X$ and $Y$ are independent random variables with uniform distribution on $[0,2^{-n}]^d$.
Therefore \eqref{eq:c4} and Lemma~\ref{l:pq} with $q=\rho(x,y)$, $g(x)-g(y)=\theta$, and $u=(t+d)/2$ yield that
\begin{align*} &\EE \left(\left(\rho(x,y)^{2}+|(f+g)(x)-(f+g)(y)|^2\right)^{-\frac{t+d}{2}}\right) \\
&=4^{nd}\int_{[0,2^{-n}]^{d}} \int_{[0,2^{-n}]^{d}} \frac{\mathrm{d} \alpha \, \mathrm{d} \beta}{(\rho(x,y)^2+|\alpha-\beta+(g(x)-g(y))|^2)^{\frac{t+d}{2}}} \\
&\leq c_1 2^{nd} \rho(x,y)^{-t}\leq c_1c_4 \rho(x,y)^{-s},
\end{align*}
so $c:=c_1 c_4$ works.
\end{proof}
Now we return to the proof of Theorem~\ref{t:H}.
By Theorem~\ref{t:energy} it is enough to prove that
$I_{t+d}(\nu_f)<\infty$ almost surely, so it is enough to show that $\EE I_{t+d}(\nu_f)<\infty$. The definition of $\nu_f$,
Fubini's theorem, Statement~\ref{st:0}, and $I_s(\nu)<\infty$ yield that
\begin{align*} \EE I_{t+d}(\nu_f)&=\EE \iint_{(\graph(f+g))^2}  \frac{\mathrm{d} \nu_f(x) \, \mathrm{d} \nu_f (y)}{\rho(x,y)^{t+d}} \\
&=\EE \iint_{K^2} \frac{\mathrm{d} \nu(x) \, \mathrm{d} \nu (y)}{\left(\rho(x,y)^{2}+|(f+g)(x)-(f+g)(y)|^2\right)^{\frac{t+d}{2}}} \\
&=  \iint_{K^2} \EE \left(\left(\rho(x,y)^{2}+|(f+g)(x)-(f+g)(y)|^2\right)^{-\frac{t+d}{2}}\right) \, \mathrm{d} \nu(x) \, \mathrm{d} \nu (y) \\
&\leq \iint_{K^2} c \rho(x,y)^{-s} \, \mathrm{d} \nu(x) \, \mathrm{d} \nu (y)=c I_s(\nu)<\infty.
\end{align*}
The proof is complete.
\end{proof}

\section{Open problems} \label{s:open}

\begin{definition} \label{d:D} Let $X$ be a Banach space. We say that the function $\Delta\colon X\to \RR$ satisfies the \emph{intertwining condition} if
for all $x,y\in X$ and Lebesgue almost every $t\in \RR$ we have
\begin{equation*}  \Delta(x-ty)\geq \Delta(y).
\end{equation*}
\end{definition}

Gruslys et al.\ \cite[Theorem~1.1]{GJMNOP} proved the following.

\begin{theorem}[Gruslys et al.] \label{t:Gea}
Let $X$ be a Banach space and let $\Delta\colon X\to \RR$ be a Borel measurable function satisfying the intertwining condition.
Then for a prevalent $x\in X$ we have
$$\Delta(x)=\sup_{y\in X} \Delta(y).$$
\end{theorem}

Fraser and Hyde~\cite{FH2} proved that if $K$ is an uncountable compact metric space then $\{f\in C(K,\RR): \dim_H f(K)=1\}$ is not only prevalent
but \emph{$1$-prevalent}, i.e.\ prevalence is witnessed by a  measure supported on a one-dimensional subspace. It would be interesting to decide whether this is a general phenomena. For more on this notion and related problems see \cite{FH2}.
Specially, we are interested whether the theorems of our paper can be generalized similarly.
It was proved in \cite{GJMNOP} that if $K\subset \RR^m$ satisfies the property of Theorem~\ref{t:GJMNOP} then the intertwining condition holds for $\Delta \colon C(K,\RR)\to \RR$, $\Delta(f)=\dim \graph(f)$, where $\dim$ denotes the upper or lower box dimension. Hence Theorem~\ref{t:Gea}
yields that prevalence can be replaced by $1$-prevalence in Theorem~\ref{t:GJMNOP}. In the case of the Hausdorff dimension the following
problem is open even for $C[0,1]$, see \cite[Question~1.5]{FH}.

\begin{problem} Let $\dim$ be one of $\dim_H$, $\underline{\dim}_B$, $\overline{\dim}_B$, or $\dim_P$.
Let $K$ be a compact metric space and let $d\in \NN^+$. Define
$$\Delta \colon C(K,\RR^d)\to \RR, \quad \Delta(f)=\dim \graph(f).$$
Does the intertwining condition hold for $\Delta$?
\end{problem}

Example~\ref{ex1} shows that Theorem~\ref{t:grM} does not remain true if $K$ may have infinitely many isolated points.
In the generic setting Hyde et al.\ \cite[Theorem~1]{HLOPS} proved the following theorem.
It was stated only for $K\subset \RR$ and $d=1$, but the proof works verbatim for the general case.
\begin{theorem}[Hyde et al.] Let $K$ be a compact metric space and let $d\in \RR^+$.
For a generic continuous function $f\in C(K, \RR^d)$ we have
\begin{equation*} \overline{\dim}_B \graph(f)=\sup_{g\in C(K,\RR^d)} \overline{\dim}_B \graph(g).
\end{equation*}
\end{theorem}
Kelgiannis and Laschos \cite{KL} explicitly computed this supremum in some non-trivial cases. It would be
interesting to know whether the analogue of the above theorem holds for prevalent maps as well.

\begin{problem} Let $K$ be a compact metric space, let $d\in \NN^+$, and let
$\dim$ be one of $\underline{\dim}_B$ or $\overline{\dim}_B$.
Is it true for a prevalent $f\in C(K,\RR^d)$ that
$$\dim \graph(f)=\sup_{g\in C(K,\RR^d)} \dim \graph(g)?$$
\end{problem}

\begin{definition}
A function $h \colon [0,\infty)\to [0,\infty)$ is defined to be a \emph{gauge function} if it is non-decreasing
and $h(0)=0$. The \emph{generalized $h$-Hausdorff measure} of a metric space $X$ is defined as
\begin{align*}
\mathcal{H}^{h}(X)&=\lim_{\delta\to 0+}\mathcal{H}^{h}_{\delta}(X)
\mbox{, where}\\
\mathcal{H}^{h}_{\delta}(X)&=\inf \left\{ \sum_{i=1}^\infty h(\diam
A_{i}): X \subseteq \bigcup_{i=1}^{\infty} A_{i},~
\forall i \, \diam A_i \le \delta \right\}.
\end{align*}
\end{definition}

This concept allows us to measure the size of metric spaces more precisely compared to Hausdorff dimension.
It is really needed to find the exact measure for the level sets of a linear Brownian motion or
for the range of a $d$-dimensional Brownian motion. For more on applications and for other references
see \cite{MP}. We are not able to decide whether the graph of a prevalent continuous map is as large as possible according to this finer scale.

\begin{problem} Let $K$ be a compact metric space and let $d\in \NN^+$. Let $h,g$ be gauge functions such that $\iH^{h}(K)>0$ and
$$\lim_{r\to 0+} \frac{g(r)}{h(r)r^d}=\infty.$$ Is it true that for a prevalent $f\in C(K,\RR^d)$ we have
$$\iH^{g} \left(\graph(f)\right)>0?$$
\end{problem}

Note that for $K\subset \RR$ with positive Lebesgue measure the above problem was answered positively in \cite{BDE}.

\subsection*{Acknowledgments}
We thank an anonymous referee for some helpful remarks and for suggesting new references.

\vspace{8mm}

\end{document}